\documentclass[12pt]{amsart}
\usepackage{mathrsfs,latexsym,amsfonts,amssymb}
\newtheorem{theorem}{Theorem}[section]
\newtheorem{lemma}[theorem]{Lemma}
\newtheorem{corollary}[theorem]{Corollary}
\newtheorem{question}[theorem]{Question}
\newtheorem{remark}[theorem]{Remark}
\theoremstyle{definition}
\newtheorem{definition}[theorem]{Definition}
\newtheorem{proposition}[theorem]{Proposition}

\begin{document}

\title[The characterizations of hyperspaces and free topological groups with an $\omega^\omega$-base]
{The characterizations of hyperspaces and free topological groups with an $\omega^\omega$-base}

\author{Fucai Lin*}
\address{Fucai Lin: 1. School of mathematics and statistics,
Minnan Normal University, Zhangzhou 363000, P. R. China; 2. Fujian Key Laboratory of Granular Computing and Application, Minnan Normal University, Zhangzhou 363000, China}
\email{linfucai@mnnu.edu.cn; linfucai2008@aliyun.com}

\author{Chuan Liu}
\address{(Chuan Liu): Department of Mathematics,
Ohio University Zanesville Campus, Zanesville, OH 43701, USA}
\email{liuc1@ohio.edu}

\thanks{Fucai Lin is supported by Fujian Provincial Natural Science Foundation of China (No.2024J02022) and the National Natural Science Foundation of China (Grant No. 11571158).}

\keywords{free topological groups; hyperspace; metrizable space; $\omega^\omega$-bases; Fr\'echet-Urysohn space}%insert keywords
\subjclass[2010]{Primary 54B20; secondary 22A05, 5430, 54D45, 54D70, 54E35}%insert subject class

%\date{\today}
\begin{abstract}
A topological space $(X, \tau)$ is said to be have an {\it $\omega^\omega$-base} if for each point $x\in X$ there exists a neighborhood base $\{U_{\alpha}[x]: \alpha\in\omega^\omega\}$ such that $U_{\beta}[x]\subset U_{\alpha}[x]$ for all $\alpha\leq\beta$ in $\omega^\omega$. In this paper, the characterization of a space $X$ is given such that the free Abelian topological group $A(X)$, the hyperspace $CL(X)$ with the Vietoris topology and the hyperspace $CL(X)$ with the Fell topology have $\omega^\omega$-bases respectively. The main results are listed as follows:

\smallskip
(1) For a Tychonoff space $X$, the free Abelian topological group $A(X)$ is a $k$-space with an $\omega^\omega$-base if and only if $X$ is a topological sum of a discrete space and a submetrizable $k_\omega$-space.

\smallskip
(2) If $X$ is a metrizable space, then $(CL(X), \tau_V)$ has an $\omega^\omega$-base if and only if $X$ is separable and the boundary of each closed subset of $X$ is $\sigma$-compact.

\smallskip
(3) If $X$ is a metrizable space, then $(CL(X), \tau_F)$ has an $\omega^\omega$-base consisting of basic neighborhoods if and only if $X$ is a Polish space.

\smallskip
(4) If $X$ is a metrizable space, then $(CL(X), \tau_F)$ is a Fr\'echet-Urysohn space with an $\omega^\omega$-base, if and only if $(CL(X), \tau_F)$ is first-countable, if and only if $X$ is a locally compact and second countable space.
\end{abstract}

\maketitle
\section{Introduction and Preliminaries}
Denote the sets of real numbers, rational numbers, positive integers and all non-negative integers by $\mathbb{R}$, $\mathbb{Q}$, $\mathbb{N}$, and $\omega$, respectively. {\bf All spaces are assumed to be Tychonoff}. Readers may refer to
\cite{AT2008,E1989,G1984} for terminology and notations not
explicitly given here.

Let $\omega^\omega$ be the set all the functions from $\omega$ to $\omega$ with the natural partial order defined
by $\alpha\leq\beta$ iff $\alpha(n)\leq \beta(n)$ for all $\alpha, \beta\in \omega^\omega$. A topological space $(X, \tau)$ is said to be have an {\it $\omega^\omega$-base} if for each point $x\in X$ there exists a neighborhood base $\{U_{\alpha}[x]: \alpha\in\omega^\omega\}$ such that $U_{\beta}[x]\subset U_{\alpha}[x]$ for all $\alpha\leq\beta$ in $\omega^\omega$.
The concept of $\omega^\omega$-base has been formally introduced in \cite{FJPS2006} for studying (DF)-spaces, C(X)-spaces in the frame of locally convex spaces. Now the class of spaces with an $\omega^\omega$-base are known in Functional Analysis and has been studied in \cite{B2017,FJPS2006,FJ2013,GKL2015,GK2016,LRZ2020}. In particular, Gabriyelyan, Kakol, Leiderman in \cite{GKL2015} studied the class of
topological groups with an $\omega^\omega$-base, and posed many interesting open questions about spaces with an $\omega^\omega$-base. Then Banakh in \cite{B2017} systematically study topological spaces with an $\omega^\omega$-base.

The main goal of the article is a thorough study of the classes of hyperspaces and free Abelian topological groups with an $\omega^\omega$-base. It is well known that hyperspaces have been the
focus of much research, see \cite{B1993, HP2002, LL2023,LL2022,M1951, NP2015}.
There are many results on the hyperspace $CL(X)$ of a topological
space $X$ equipped with various topologies. In this paper, we endow
$CL(X)$ with the Vietoris topology $\tau_{V}$ and the Fell topology
$\tau_{F}$ respectively.

Let $X$ be a space. For each $n\in\mathbb{N}$, let
$\mathcal{F}_n(X)=\{A\subset X: 1\leq|A|\leq n\}$ and $\mathcal{F}(X)=\{A\subset X: A\neq\emptyset\ \mbox{and}\ A\ \mbox{is finite}\};$
 moreover, let
$$CL(X)=\{A\subset X: A \ \mbox{is nonempty closed in}\ X\}$$ and $$\mathcal{K}(X)=\{K\subset X: K\ \mbox{is nonempty compact}\}.$$

\smallskip
If $\mathcal{U}$ is a nonemptyset family of subsets of a space $X$, let
$$\mathcal{U}^-=\{A\in CL(X): A\cap U\neq\emptyset, U\in
\mathcal{U}\}$$ and $$\mathcal{U}^+=\{A\in CL(X): A\subset
\bigcup\mathcal{U}\};$$in particular, if $\mathcal{U}=\{U\}$, then we denote $\mathcal{U}^-$ and $\mathcal{U}^+$ by $U^-$ and $U^+$ respectively.

For any finite family $\mathcal{U}=\{U_1, \ldots, U_k\}$ of open subsets of space $X$, let $$\langle\mathcal{U}\rangle=\langle
U_1, ..., U_k\rangle =\{H\in CL(X): H\subset {\bigcup}_{i=1}^k U_i\
\mbox{and}\ H\cap U_j\neq \emptyset, 1\leq j\leq k\}.$$ We endow
$CL(X)$ with {\it Vietoris topology} defined as the topology
generated by $$\{\langle U_1, \ldots, U_k\rangle: U_1, \ldots, U_k\
\mbox{are open subsets of}\ X, k\in \mathbb{N}\}.$$  We endow
$CL(X)$ with {\it Fell topology} defined as the topology generated
by the following two families as a subbase:
$$\{U^{-}: U\ \mbox{is any non-empty open in}\ X\}$$ and $$\{(K^{c})^{+}: K\ \mbox{is a compact subset of}\ X, K\neq X\}.$$
The {\it locally finite topology} $\tau_{loc fin}$
on $CL(X)$ has a subbase consisting of all subsets of the forms $V^+$
and $\mathcal{U}^-$, where $V$ ranges over all open subsets of $X$
and $\mathcal{U}$ range over all the locally finite families of open
subsets of $X$. We denote the hyperspace $CL(X)$ with the Vietoris topology, the Fell
topology and the locally finite topology by $(CL(X), \tau_{V})$, $(CL(X), \tau_{F})$ and $(CL(X), \tau_{lf})$
respectively.

\smallskip
In this paper, we mainly consider the following four questions.

\begin{question}\label{q0}
How to characterize a metrizable space $X$ such that $(CL(X), \tau_V)$ has an $\omega^\omega$-base?
\end{question}

\smallskip
\begin{question}\label{q1}
How to characterize a metrizable space $X$ such that $(CL(X), \tau_F)$ has an $\omega^\omega$-base?
\end{question}

\smallskip
\begin{question}\label{q2}
How to characterize a space $X$ such that free Abelian topological group $A(X)$ has an $\omega^\omega$-base?
\end{question}

A space $X$ is said to be {\it Fr\'{e}het-Urysohn } if for each point $x\in X$ and each set $A\subset X$ with $x\in \overline{A}$ there is a sequence $(a_{n})_{n\in\omega}$ in $A$ converging to $x$. In \cite{GKL2015}, Gabriyelyan, K\c{a}kol and Leiderman proved that each Fr\'{e}het-Urysohn topological group having an $\omega^\omega$-base is first-countable. Therefore, the following question is interesting.

\smallskip
\begin{question}\label{q4}
For a space $X$, if $(CL(X), \tau_F)$ (resp., $(CL(X), \tau_V)$) is a Fr\'{e}het-Urysohn space having an $\omega^\omega$-base, is $(CL(X), \tau_F)$ (resp., $(CL(X), \tau_V)$) first-countable?
\end{question}

In this paper, we first give an answer for Question~\ref{q2} in Section 2, and prove that the free Abelian topological group $A(X)$ over a Tychonoff space is a $k$-space with an $\omega^\omega$-base if and only if $X$ is a topological sum of a discrete space and a submetrizable $k_\omega$-space. In Section 3, we give full answers for Questions~\ref{q0} and~\ref{q1}, see Theorems~\ref{tt1} and~\ref{tt2}. Moreover, we give an affirmative answer for Question~\ref{q4} in the case of Fell topology, see Theorem~\ref{tt4}.

Next, we give some necessary notations and terminologies.

Let $X$ be a topological space and $P \subseteq X$ be a subset of $X$. The \emph{closure} of $P$ in $X$ is denoted by $\overline{P}$, and $\partial P$ denote the boundary of $P$ in $X$. The set $P$ is called a {\it sequential neighborhood} \cite{FR1965} of $x$ in $X$, if each sequence
converging to $x\in X$ is eventually in $P$. A subset
$U$ of $X$ is called {\it sequentially open} if $U$ is a sequential neighborhood of each of
its points. The space $X$ is called a {\it sequential space} if each sequentially open subset of $X$ is
open. The space $X$ is called a \emph{$k$-space} provided that a subset $C\subseteq X$ is closed in $X$ if $C\cap K$ is closed in $K$ for each compact subset $K$ of $X$.
If there exists a family of countably many compact subsets $\{K_{n}: n\in\mathbb{N}\}$ of
$X$ such that each subset $F$ of $X$ is closed in $X$ provided that $F\cap K_{n}$ is closed in $K_{n}$ for each $n\in\mathbb{N}$, then $X$ is called a \emph{$k_{\omega}$-space}. Note that every $k_{\omega}$-space is a $k$-space.

\begin{definition}
 Let $\mathcal{P}$ be a family of subsets of
a space $X$. Then $\mathcal{P}$ is called a {\it $cs^{\ast}$-network} \cite{LT1994}
of $X$ if whenever a sequence $\{x_{n}\}$ converges to $x\in U\in\tau
(X)$, there is a $P\in\mathcal{P}$ and a subsequence $\{x_{n_{i}}: i\in \mathbb{N}\}\subset\{x_{n}: n\in \mathbb{N}\}$ such that $\{x\}\cup\{x_{n_{i}}: i\in \mathbb{N}\}\subset P\subset U$.
\end{definition}

\begin{definition}
For a topological space $X$ its {\it free topological group} (resp., {\it free Abelian topological group, free locally convex space}) is a pair $(F(X), h_{X})$ (resp., $(A(X), h_{X})$, $(L(X), h_{X}))$ consisting
of a topological group $F(X)$ (resp., topological Abelian group $A(X)$, locally convex topological vector space $L(X)$) and a continuous map
$h_{X}: X\rightarrow F(X)$ (resp, $h_{X}: X\rightarrow A(X)$,  $h_{X}: X\rightarrow L(X)$) such that for every continuous map $f: X\rightarrow G$ into a topological group (resp., topological Abelian group, locally convex topological vector space) $G$ there exists a continuous homomorphic $\widetilde{f}: F(X)\rightarrow G$ (resp., continuous homomorphic $\widetilde{f}: A(X)\rightarrow G$, continuous linear operator $\widetilde{f}: L(X)\rightarrow G$) such that $\widetilde{f}\circ h_{X}=f$.
\end{definition}

Let $X$ be a space and $\kappa$ be an infinite cardinal. For each $\alpha\in\kappa$, let $T_{\alpha}$ be a sequence converging to
  $x_{\alpha}\not\in T_{\alpha}$ in $X$. Let $T=\bigoplus_{\alpha\in\kappa}(T_{\alpha}\cup\{x_{\alpha}\})$ be the topological sum of $\{T_{\alpha}
  \cup \{x_{\alpha}\}: \alpha\in\kappa\}$. Then
  $S_{\kappa} =\{\infty\}  \cup \bigcup_{\alpha\in\kappa}T_{\alpha}$
  is the quotient space obtained from $T$ by
  identifying all the points $x_{\alpha}\in T$ to the point $\infty$.

\maketitle
\section{free Abelian topological groups with an $\omega^\omega$-base}
In this section, we mainly characterise the space $X$ such that the free Abelian topological groups has an $\omega^\omega$-base. Indeed, Lin, Ravsky and Zhang in \cite{LRZ2020} proved the following theorem.

\begin{theorem}\cite[Theorem 4.6]{LRZ2020}
For a space $X$, the free topological group $F(X)$ is a $k$-space with an $\omega^\omega$-base if and only if $X$ is either countable discrete or a submetrizable $k_{\omega}$-space.
\end{theorem}

They also pointed out in \cite{LRZ2020} that $A(C\oplus D)$ is a $k$-space with an $\omega^\omega$-base, where $C$ is a non-trivial convergent sequence with its limit point and $D$ is an uncountable discrete space; however, $C\oplus D$ is neither discrete nor Lindel\"{o}f. The following Theorem~\ref{tt3} provides a characterization of a space $X$ such that the free Abelian topological group $A(X)$ is a $k$-space with an $\omega^\omega$-base, which gives some partial answers for Question~\ref{q2} and \cite[Questions 4.15 and 4.17]{GKL2015} respectively.

\begin{theorem}\label{tt3}
Let $A(X)$ be an free Abelian topological group over a space $X$. Then $A(X)$ is a $k$-space with an $\omega^\omega$-base if and only if $X$ is a topological sum of a discrete space and a submetrizable $k_\omega$-space.
\end{theorem}

\begin{proof}
Necessity. Assume that $X=D\oplus Y$, where $D$ is discrete and $Y$ is a submetrizable $k_\omega$-subspace. Then $A(X)$ is topologically homeomorphic to $A(D)\times A(Y)$. Since $A(D)$ is discrete and $A(Y)$ is a submetrizable $k_\omega$-space, it follow that both $A(D)$ and $A(Y)$ have $\omega^\omega$-bases by \cite[Theorem 1.5]{GKL2015}. Then $A(D)\times A(Y)$ has an $\omega^\omega$-base; therefore, $A(X)$ has an $\omega^\omega$-base.

Sufficiency. Assume that $A(X)$ is a $k$-space with an $\omega^\omega$-base. By \cite[Corollary 3.13]{GKL2015}, $A(X)$ is metrizable or contains a submetrizable open $k_\omega$-subgroup. If $A(X)$ is metrizable, then $X$ is discrete. If $A(X)$ is not metrizable, then $A(X)$ contains a submetrizable open $k_\omega$-subgroup, hence $A(X)=\oplus_{\alpha\in \Gamma}G_\alpha$, where each $G_\alpha$ is a submetrizable $k_\omega$-space. Since $X$ is a closed subset of $A(X)$ and $X=X\cap A(X)=\oplus_{\alpha\in \Gamma}(X\cap G_\alpha)$, we conclude that each $X\cap G_\alpha$ is a submetrizable $k_\omega$-space. Put $$\Gamma_{1}=\{\alpha\in \Gamma: X\cap G_\alpha\ \mbox{is not discrete}\}.$$Now it suffices to prove the set $\Gamma_{1}$ is countable. Indeed, if the set $\Gamma_{1}$ is countable, then we can rewrite $\Gamma_{1}=\{X\cap G_{\alpha_i}: i\in \mathbb{N}\}$; hence $X=D\oplus (\oplus_{i\in\mathbb{N}}(X\cap G_{\alpha_i}))$, where $D$ is discrete and each $X\cap G_{\alpha_i}$ is a non-discrete $k_\omega$-space.

Suppose that $\Gamma_{1}$ is not countable, next we shall give a contradiction. Indeed, we may assume that $X$ contains a subspace $\oplus_{\alpha\in \omega_1} (X\cap G_\alpha)$, where each $X\cap G_\alpha$ is not discrete. Obviously, each $X\cap G_\alpha$ is sequential. Hence, for each $\alpha\in \omega_1$, there is a nontrivial sequence $$\{x(\alpha)\}\cup\{x_n(\alpha): n\in \mathbb{N}\}\subset X\cap G_\alpha$$ with $x_n(\alpha)\to x(\alpha)$ as $n\rightarrow\infty$; then $x_n(\alpha)-x(\alpha)\to 0$ as $n\to \infty$. Put $$Z=\{x_n(\alpha)-x(\alpha): n\in\mathbb{N}, \alpha\in \omega_1\}\cup\{0\}.$$ Then it is easy to see that $Z$ is a copy of $S_{\omega_1}$, which shows that $S_{\omega_1}$ has an $\omega^\omega$-base. This contradicts to \cite[Corollary 8.2.3]{B2017}.
\end{proof}

In \cite{GKL2015}, Gabriyelyan, K\c{a}kol and Leiderman posed the following question.

\begin{question}\cite[Question 4.19]{GKL2015}\label{q3}
Let $X$ be a space. If $A(X)$ has an $\omega^\omega$-base, does the free locally convex space $L(X)$ have an $\omega^\omega$-base?
\end{question}

By Theorem~\ref{tt3} and \cite[Corollary 8.3]{BL2016}, $A(X)$ has an $\omega^\omega$-base and $L(X)$ does not have an $\omega^\omega$-base for any space $X$ which is a topological sum of a uncountable discrete space $D$ with a metrizable $k_\omega$-space $Y$, which gives a negative answer to Question~\ref{q3}. However, by \cite[Corollary 8.3]{BL2016}, it follows that, for each metrizable space $X$, the free locally convex space $L(X)$ has an $\omega^\omega$-base if and only if $X$ is $\sigma$-compact. Therefore, $A(\mathbb{Q})$ has an $\omega^\omega$-base, but $A(\mathbb{Q})$ is not a $k$-space by Theorem~\ref{tt3}. Hence the assumption ``$k$-space'' in Theorem~\ref{tt3} is not a necessary condition such that $A(X)$ has an $\omega^\omega$-base. However, we have the following two corollaries.

\begin{corollary}\label{cc0}
Let $X$ be a locally compact separable metrizable space. Then $A(X)$, $F(X)$ and $L(X)$ have $\omega^\omega$-bases respectively.
\end{corollary}

\begin{proof}
Since each locally compact separable metrizable space is a $k_{\omega}$-space, it follows from \cite[Theorems 7.6.30 and 7.6.36]{AT2008} that $F(X)$ and $A(X)$ are $k$-spaces. Now we apply Theorem~\ref{tt3}, \cite[Corollary 8.3]{BL2016} and \cite[Theorem 4.4]{LRZ2020}, the results are desired.
\end{proof}

The following corollary is obvious by Corollary~\ref{cc0}.

\begin{corollary}\label{c0}
For any $n\in\mathbb{N}$, $F(\mathbb{R}^{n})$, $A(\mathbb{R}^{n})$ and $L(\mathbb{R}^{n})$ have $\omega^\omega$-bases, where $\mathbb{R}$ is the real number space.
\end{corollary}

\maketitle
\section{The hyperspace $CL(X)$ with an $\omega^\omega$-base}
In this section, we mainly characterise metrizable space $X$ such that $(CL(X), \tau_V)$ and $(CL(X), \tau_F)$ have an $\omega^\omega$-base respectively. First, we give a characterization of a space $X$ such that $(\mathcal{F}(X), \tau_V)$ has an $\omega^\omega$-base.

\begin{proposition}\label{pr6}
Let $X$ be a space. Then $(\mathcal{F}(X), \tau_V)$ has an $\omega^\omega$-base if and only if $X$ has an $\omega^\omega$-base.
\end{proposition}

\begin{proof}
Obviously, it suffices to prove the necessity. Assume that $X$ has an $\omega^\omega$-base. Take any ${\bf x}=\{x_i: i\leq n\}\in \mathcal{F}(X)$. For each $i\leq n$, let $\mathcal{B}_i=\{B_\alpha(i): \alpha\in \omega^\omega\}$ be an $\omega^\omega$-base at $x_i$. Put $${\bf \mathcal {B}}=\{\langle B_\alpha(1), ..., B_\alpha(n)\rangle: B_\alpha(i)\in \mathcal{B}_i, i\leq n, \alpha\in \omega^\omega\}.$$ We claim that ${\bf \mathcal {B}}$ is an $\omega^\omega$-base at ${\bf x}$ in $(\mathcal{F}(X), \tau_V)$. Indeed, we divide the proof into the following two steps.

\smallskip
(1) Take any basic open neighborhood ${\bf U}$ of ${\bf x}$. Without loss of generality, we may assume that ${\bf U}=\langle U_1, ..., U_n\rangle$ of ${\bf x}$ such that  for each $i\leq n$ we have $x_i\in U_i$ and $U_i\cap U_j=\emptyset$ if $j\neq i$. For each $i\leq n$, choose an $\alpha_i\in\omega^\omega$ such that $x_i\in B_{\alpha_i}(i)\subset U_i$, where $B_{\alpha_i}(i)\in \mathcal{B}_i$. Hence $\langle B_{\alpha_1}(1), ..., B_{\alpha_n}(n)\rangle\subset {\bf U}$. Take an $\alpha\in\omega^\omega$ such that $\alpha(j)=\max\{\alpha_i(j): i\leq n\}$ for each $j\in\omega$. Then, for each $i\leq n$, we have $\alpha_{i}\leq\alpha$, hence $B_\alpha(i)\subset B_{\alpha_i}(i)$. Therefore, $$\langle B_\alpha(1), ..., B_\alpha(n)\rangle\subset \langle B_{\alpha_1}(1), ..., B_{\alpha_n}(n)\rangle\subset {\bf U}.$$

\smallskip
(2) Take any $\alpha, \beta\in \omega^\omega$ such that $\alpha <\beta$. Then it is easy to see that $$\langle B_\beta(1), ..., B_\beta(n)\rangle\subset \langle B_\alpha(1), ..., B_\alpha(n)\rangle$$ since $B_{\beta}(i)\subset B_\alpha(i)$ for each $i\leq n$. Hence $\mathcal{B}$ is an $\omega^\omega$-base in $(\mathcal{F}(X), \tau_V)$.
\end{proof}

By the proof of Proposition~\ref{pr6}, it follows that $(\mathcal{F}(X), \tau_V)$ has an $\omega^\omega$-base consisting
of basic neighborhoods if $X$ has an $\omega^\omega$-base. However, the situation is different for the case of $(\mathcal{F}(X), \tau_F)$. In order to see that we first give a technical lemma.

\begin{lemma}\label{l01}
Let $B\subset \omega^\omega$ with $|B|=\omega_1$. Then there exists an $\alpha\in \omega^\omega$ such that $|\{\beta\in B: \beta\leq \alpha\}|\geq \omega$.
\end{lemma}

\begin{proof}
Since $|B|=\omega_1$, there exists an $n_0\in\omega$ such that $|\{\beta\in B: \beta(0)=n_0\}|=\omega_1$. Put $$B_0=\{\beta\in B: \beta(0)=n_0\}.$$ Then we can find the smallest natural number $i_1>0$ such that $$|\{\beta\in B_0: \beta(i_1)= n_1\}|=\omega_1$$ for some $n_1\in\mathbb{N}$ and there exist $\beta_{11}\in B_0, \beta_{12}\in \{\beta\in B_0: \beta(i_1)=n_1\}$ with $\beta_{11}(i_1)\neq \beta_{12}(i_1)=n_1$. Put $$B_1=\{\beta\in B_0: \beta(i_1)=n_1\}.$$By induction, we can obtain a countable family $\{B_k: k\in\omega\}$ of subsets of $B$ and an increasing sequence  $\{i_k: k\in\mathbb{N}\}$ which satisfy the following two conditions (i) and (ii):

\smallskip
(i) $B_{k+1}\subset B_{k}\subset B$ for any $k\in\omega$, where $B_k=\{\beta\in B_{k-1}: \beta(i_k)=n_k\}$ and $|B_k|=\omega_1$ for each $k\in\omega$ and $B_{-1}=B$;

\smallskip
(ii) An infinite countable subset $\{\beta_{k1}\in B_{k-1}: k\in \mathbb{N}\}\cup \{\beta_{k2}\in B_k: k\in \mathbb{N}\}$ of $\omega^\omega$ with $\beta_{k1}(i_k)\neq \beta_{k2}(i_k)=n_k$ for any $k\in\mathbb{N}$.

\smallskip
 Define $\alpha\in\omega^\omega$ by $\alpha(k)=\sum_{i=1}^k(\beta_{i1}(i)+\beta_{i2}(i))$ for each $k\in\omega$. Obviously, we have $\beta_{k1}\leq \alpha$ and $\beta_{k2}\leq \alpha$ for each $k\in\omega$. The proof is completed.
\end{proof}

\begin{proposition}\label{pr7}
Let $D$ be a discrete space with $|D|=\omega_1$. Then $(\mathcal{F}(D), \tau_F)$ does not have any $\omega^\omega$-base.
\end{proposition}

\begin{proof}
Fix any $x\in D$. We prove that $(\mathcal{F}(D), \tau_F)$ does not have any $\omega^\omega$-base at $\{x\}$. If not, we suppose that $\mathcal{B}=\{{\bf U}_\alpha: \alpha\in \omega^\omega\}$ is an $\omega^\omega$-base at $\{x\}$ in $(\mathcal{F}(D), \tau_F)$. Enumerate $D\setminus\{x\}$ as $\{d_\beta: \beta\in\omega_1\}$. For each $\beta\in\omega_1$, the set $(\{d_\beta\}^c)^+$ is an open neighborhood of $\{x\}$, hence we can choose an $\alpha_\beta\in\omega^\omega$ such that ${\bf U}_{\alpha_\beta}\subset (\{d_\beta\}^c)^+$. Let $B=\{\alpha_\beta\in\omega^\omega: \beta\in \omega_1\}$. We divide the proof into the following two cases.

\smallskip
 {\bf Case 1:} $|B|<\omega_1$.

\smallskip
 There exist $\beta_0\in\omega_1$ and a uncountable subset $B_1\subset B$ such that $\alpha_\gamma=\alpha_{\beta_0}$ for any $\gamma\in B_1$. Then ${\bf U}_{\beta_0}\subset (\{d_\gamma\}^c)^+$ for each $\gamma\in B_1$. Choose a basic neighborhood $\mathcal{V}^-\cap (H^c)^+$ of $\{x\}$ such that $\mathcal{V}^-\cap (H^c)^+\cap \mathcal{F}(D)\subset {\bf U}_{\beta_0}\subset (\{d_{\gamma}\}^c)^+$ for each $\gamma$, where $\mathcal{V}$ is a finite family of open subsets and $H$ is compact in $X$. Then $H\supset \{d_{\gamma}: \gamma\in B_1\}$, which implies that $H$ an infinite compact subset of $D$. This is a contradiction.

\smallskip
 {\bf Case 2:} $|B|=\omega_1$.

\smallskip
  By Lemma ~\ref{l01}, there exists $\gamma\in\omega^\omega$ such that $\{\alpha_\beta\in B: \alpha_\beta\leq \gamma\}$ is countable and infinite. Enumerate $\{\alpha_\beta\in B: \alpha_\beta\leq \gamma\}$ as $\{\alpha_{\beta(i)}: i\in\omega\}$. Then $$\{x\}\in {\bf U}_\gamma\subset {\bf U}_{\alpha_{\beta(i)}}\subset (\{d_{\beta(i)}\}^c)^+$$for each $i\in\omega$. Choose a basic neighborhood $\mathcal{U}^-\cap (K^c)^+$ of $\{x\}$ such that $$\mathcal{U}^-\cap (K^c)^+\cap \mathcal{F}(D)\subset {\bf U}_\gamma\subset (\{d_{\beta(i)}\}^c)^+$$ for each $i\in\omega$, where $\mathcal{U}$ is a finite family of open subsets and $K$ is compact in $X$. Hence $K\supset \{d_{\beta(i)}\}$ for each $i\in\omega$, then $K$ is an infinite compact subset of $D$, which is a contradiction.

  Therefore, $(\cap \mathcal{F}(D), \tau_F)$ does not have any $\omega^\omega$-base.
\end{proof}

By Proposition~\ref{pr7}, we have the following corollary.

\begin{corollary}
Let $D$ be a discrete space with $|D|=\omega_1$. Then $(CL(D), \tau_F)$ does not have an $\omega^\omega$-base.
\end{corollary}

Next, we give an answer for Question~\ref{q0}. First, we give some technical lemmas.

\begin{lemma}\label{le1}
\cite[Lemma 2.3.1]{M1951} Let $U_1, ..., U_n$ and $V_1,..., V_m$ be
subsets of a space $X$. Then, in Vietoris topology $(CL(X),
\tau_V)$, we have $\langle U_1, ..., U_n\rangle\subset \langle
V_1,..., V_m\rangle$ if and only if $\bigcup_{j=1}^nU_j\subset
\bigcup_{j=1}^mV_j$ and for every $V_i$ there exists an $U_k$ such
that $U_k\subset V_i$.
\end{lemma}

The following lemma is a complement for Propositions~\ref{pr6} and~\ref{pr7}.

\begin{lemma}\label{l1}
Let $D$ be a discrete space with $|D|=\omega_1$. Then $(CL(D), \tau_V)$ does not have any $\omega^\omega$-base.
\end{lemma}

\begin{proof}
Suppose that $(CL(D), \tau_V)$ have an $\omega^\omega$-base. Let $\mathcal{B}=\{B_\alpha: \alpha\in\omega^\omega\}$ be an $\omega^\omega$-base at $D\in CL(D)$. Enumerate $D$ as $\{d_\beta: \beta\in \omega_1\}$. For each $\beta\in \omega_1$, it is obvious that $\langle \{d_\beta\}, D\rangle$ is a neighborhood of $D$, hence there is $B_{\alpha_\beta}\in \mathcal{B}$ such that $D\in B_{\alpha_\beta}\subset \langle \{d_\beta\}, D\rangle$.
If $|\{\alpha_\beta: \beta\in \omega_1\}|<\omega_1$, then there are infinitely many $\alpha_\beta$ are equal; then let $\gamma$ be such an $\alpha_\beta$.
If $|\{\alpha_\beta: \beta\in \omega_1\}|=\omega_1$, by Lemma ~\ref{l01}, there is a $\gamma\in\omega^\omega$ such that $\gamma$ is greater or equal than infinitely many $\alpha_\beta$.
Then we have $$B_\gamma\subset B_{\alpha_\beta}\subset \langle \{d_\beta\}, D\rangle$$ for infinitely many $\beta$. Choose a basic neighborhood $\langle V_1, ..., V_n, D\rangle\subset B_\gamma$ of $D$;  by Lemma ~\ref{le1}, each $\{d_\beta\}$ contains some $V_i$, which implies $\{d_\beta\}=V_i$. However, this is a contradiction since there are infinitely many such $\beta$ and only finitely many $V_{i}$.
\end{proof}

However, the situation is different for the locally finite topology.

\begin{proposition}
Let $D$ be a discrete space. Then $(CL(D), \tau_{lf})$ is discrete, hence it has an $\omega^\omega$-base.
\end{proposition}

\begin{proof}
Take any $A\in CL(D)$. Then $A$ is open in $D$, hence $\{A\}=A^{+}\cap (\bigcap_{a\in A}\{a\}^{-})$ is an open neighborhood of $A$. Therefore, $(CL(D), \tau_{lf})$ is discrete, thus it has an $\omega^\omega$-base.
\end{proof}

Let $f: X\rightarrow Y$ be a map between two spaces $X$ and $Y$. Then $f$ is said to be {\it closed} if $f(F)$ is closed in $Y$ for each closed subset $F$ in $X$.

\begin{lemma}\label{l2}
Let $X$ be a metrizable space, and let $A\in CL(X)$ such that $A$ is nowhere dense. Then $A$ has an $\omega^\omega$-base in $X$ if and only if $A$ is $\sigma$-compact.
\end{lemma}

\begin{proof}
Sufficiency. Suppose $A$ is $\sigma$-compact. Then we can assume that $A=\bigcup_{n\in\omega}K_n$, where each $K_n$ is compact. Since $X$ is metrizable, for each $n\in\mathbb{N}$, let $\{B_m(n): m\in\omega\}$ be a decreasingly local base at $K_n$ in $X$ such that $B_{m+1}(n)\subset B_m(n)$ for each $m\in\mathbb{N}$. For each $\alpha\in \omega^\omega$, let $U_\alpha=\bigcup_{n\in\omega}B_{\alpha(n)}(n)$. Then $\{U_\alpha: \alpha\in \omega^\omega\}$ is an $\omega^\omega$-base at $A$ in $X$.

Necessity. Let $A$ have an $\omega^\omega$-base $\{U_\alpha: \alpha\in \omega^\omega\}$ in $X$. Let $Y$ be the quotient space of $X$ by identifying $A$ as a point $p$. Clearly, the quotient map $f: X\to Y$ is a closed map. In order to see that $A$ is $\sigma$-compact, by \cite[Corollary 8.2.3(3)]{B2017}, it suffices to prove that $Y$ has an $\omega^\omega$-base.

For each $\alpha\in\omega^\omega$, let $V_\alpha=Y\setminus f(X\setminus U_\alpha)$. We claim that $\{V_\alpha: \alpha\in \omega^\omega\}$ is an $\omega^\omega$-base at $p$. Indeed, it is easy to check that $V_\beta\subset V_\alpha$ if $\alpha<\beta$. Now pick any open set $W$ with $p\in W$ in $X$; then $f^{-1}(p)=A\subset f^{-1}(W)$, hence there is $U_\alpha$ such that $A\subset U_\alpha\subset f^{-1}(W)$, then since $X\setminus f^{-1}(W)\subset X\setminus U_\alpha$,  it follows that $$p\in V_\alpha=Y\setminus f(X\setminus U_\alpha)\subset Y\setminus f(X\setminus f^{-1}(W))=Y\setminus(Y\setminus W)=W.$$ Hence $\{V_\alpha: \alpha\in \omega^\omega\}$ is an $\omega^\omega$-base at $p$.

For any $y\in Y\setminus \{p\}$, we have $y\in X\setminus A$. Let $\{U_\alpha(y): \alpha\in \omega^\omega\}$ be an $\omega^\omega$-base at $y$ in $X$. Then $\{U_\alpha(y)\cap (X\setminus A): \alpha\in \omega^\omega\}$ is also an $\omega^\omega$-base at $y$ in $X$. For each $\alpha\in \omega^\omega$, let $V_\alpha(y)=U_\alpha(y)\cap (X\setminus A)$. Therefore, $\{V_\alpha: \alpha\in \omega^\omega\}$ is an $\omega^\omega$-base at $y$ in $Y$.

 Therefore, $Y$ has an $\omega^\omega$-base.
\end{proof}

Now we can prove the second main theorems.

\begin{theorem}\label{tt1}
Let $X$ be a metrizable space. Then $(CL(X), \tau_V)$ has an $\omega^\omega$-base if and only if $X$ is separable and the boundary of each closed subset of $X$ is $\sigma$-compact.
\end{theorem}

\begin{proof}
Necessity. Suppose $(CL(X), \tau_V)$ has an $\omega^\omega$-base. First, we claim that $X$ is separable. Suppose not, there is a closed discrete subset $D\subset X$ with $|D|=\omega_1$. Enumerate $D$ as $\{x_\alpha: \alpha<\omega_1\}$. Since $(CL(D), \tau_V)$ is a closed subspace of $(CL(X), \tau_V)$, it follows that $(CL(D), \tau_V)$ has an $\omega^\omega$-base; however, by Lemma~\ref{l1}, $(CL(D), \tau_V)$ does not have any $\omega^\omega$-base, which is a contradiction. Hence $X$ is separable. Next we prove that the boundary of each closed subset of $X$ is $\sigma$-compact.

Fix any $A\in CL(X)$. We prove that $\partial A$ is $\sigma$-compact. Without loss of generality, we may assume that $\partial A$ is nonempty. By Lemma~\ref{l2}, it suffices to prove that $\partial A$ has an $\omega^\omega$-base in $X$. Since $\{\partial A\}\in CL(X)$, we conclude that $(CL(X), \tau_V)$ has an $\omega^\omega$-base $\{{\bf U}_\alpha: \alpha\in \omega^\omega\}$ at $\{\partial A\}$. For each $\alpha\in\omega^\omega$, assume that $${\bf U}_\alpha=\bigcup\{\langle \mathcal{U_\gamma}(\alpha)\rangle: \gamma\in \Gamma_\alpha\},$$  where each $\mathcal{U_\gamma}(\alpha)$ is a finite family of open subsets of $X$; let $$U_\alpha=\bigcup\{\bigcup\mathcal{U}_\gamma(\alpha): \gamma
\in \Gamma_{\alpha}\}.$$ In order to complete the proof, we need to prove the following Claim 1.

\smallskip
{\bf Claim 1.} The family $\{U_\alpha: \alpha\in  \omega^\omega\}$ is an $\omega^\omega$-base for $\partial A$ in $X$.

\smallskip
Indeed, take any $\alpha, \beta\in\omega^\omega$ with $\alpha<\beta$. Then since ${\bf U}_\beta\subset {\bf U}_\alpha$, we conclude that $U_\beta\subset U_\alpha$. Suppose not, there exists $x\in U_\beta\setminus U_\alpha$ such that $x$ belongs to some $U'\in \mathcal{U}_\gamma(\beta)$, where $\gamma\in\Gamma_{\beta}$. Enumerate $\mathcal{U}_\gamma(\beta)$ as $\{U_i(\gamma): i\leq n_\gamma\}$. For any $i\leq n_\gamma$, pick $x_i\in U_i(\gamma)$; let $B=\{x\}\cup\{x_i: i\leq n_\gamma\}$. Then $B\in {\bf U}_\beta\setminus {\bf U}_\alpha$, which is a contradiction.

Let $\partial A\subset V$ with $V$ open in $X$. Then $\langle V\rangle$ is an open neighborhood of $\{\partial A\}$ in $(CL(X), \tau_V)$, hence there exists an $\alpha\in\omega^\omega$ such that $\{\partial A\}\in {\bf U}_\alpha\subset \langle V\rangle$. Then it is straightforward to prove that $U_\alpha\subset V$ by Lemma~\ref{le1}.

Sufficiency. Suppose $X$ is separable and the boundary of each closed subset of $X$ is $\sigma$-compact. Fix any $A\in CL(X)$. Then $A$ is separable and $\partial A$ is $\sigma$-compact. Hence $\partial A$ has an $\omega^\omega$-base $\{U_\alpha: \alpha\in \omega^\omega\}$ in $X$ by Lemma ~\ref{l2}, then $\{U_\alpha\cup \mbox{int}(A): \alpha\in \omega^\omega\}$ is an $\omega^\omega$-base at $A=\partial A\cup \mbox{int}(A)$ in $X$. Without loss of generality, we may assume $\{U_\alpha: \alpha\in \omega^\omega\}$ is an $\omega^\omega$-base at $A$ in $X$.

Since $A$ is separable, let $D=\{x_n: n\in\omega\}$ be a countable dense subset of $A$. For each $n\in\mathbb{N}$, let $\{B_m(n): m\in\omega\}$ be a countable, decreasingly local base at $x_n$ in $X$. For each $\alpha\in\omega^\omega$,
put $${\bf U}_\alpha=\langle B_{\alpha(1)}(1)\cap U_{\alpha},  B_{\alpha(1)}(2)\cap U_{\alpha},..., B_{\alpha(1)}(\alpha(1))\cap U_{\alpha}, U_{\alpha}\rangle;$$ then ${\bf U}_\alpha$ is a neighborhood of $A$. Next, it suffices to prove the following Claim 2.

\smallskip
{\bf Claim 2.} The family $\{{\bf U}_\alpha: \alpha\in \omega^\omega\}$ is an $\omega^\omega$-base at $A$ in $(CL(X), \tau_V)$.

\smallskip
We divide the proof into the following two steps (a) and (b).

\smallskip
(a) Take any $\alpha, \beta\in\omega^\omega$ with $\alpha<\beta$. Then $$U_{\beta}\subset U_{\alpha}, \alpha(1)\leq\beta(1), B_{\beta(1)}(i)\subset B_{\alpha(1)}(i)$$  for each $i\leq \alpha(1)$. By Lemma~\ref{le1}, we have $$\langle B_{\beta(1)}(1)\cap U_{\beta},  B_{\beta(1)}(2)\cap U_{\beta},..., B_{\beta(1)}(\alpha(1))\cap U_{\beta}, ..., B_{\beta(1)}(\beta(1))\cap U_{\beta}, U_{\beta}\rangle\subset$$ $$\langle B_{\alpha(1)}(1)\cap U_{\alpha},  B_{\alpha(1)}(2)\cap U_{\alpha},..., B_{\alpha(1)}(\alpha(1))\cap U_{\alpha}, U_{\alpha}\rangle,$$ that is, ${\bf U}_\beta\subset {\bf U}_\alpha$.

\smallskip
(b) Let ${\bf W}=\langle W_1, ..., W_k\rangle$ be an open neighborhood of $A$ in $(CL(X), \tau_V)$, where each $W_{i}$ is open in $X$. Then there exists $\alpha\in\omega^\omega$ such that $U_\alpha \subset \bigcup_{i\leq k}W_i$. Pick any $x_{j_i}\in W_{i}\cap D$ for each $i\leq k$. For each $i\leq k$, there is an $m_i\in\mathbb{N}$ such that $x_{j_i}\in B_{m_i}(j_i)\subset W_i$. Let $m'=\max\{m_1, ..., m_k, \alpha(1), j_1, ..., j_k\}$. Pick $\gamma\in \omega^\omega$ with $\gamma(1)=m', \gamma(i)=\alpha(i)$ for any $i\geq 2$; then $\alpha\leq \gamma$.
 We conclude that ${\bf U_\gamma}\subset {\bf W}$. Indeed, we have $U_\gamma\subset U_\alpha\subset \bigcup_{i\leq k}W_i$. Clearly, each $W_i$ contains $B_{m_i}(j_i)$ and each $B_{m_i}(j_i)$ contains $B_{\gamma(1)}(j_i)$ since $\gamma(1)\geq m_i$ for each $i\leq k$. Note that  $\gamma(1)\geq j_i$ for each $i\leq k$, it is straightforward to prove $${\bf U_\gamma}=\langle B_{\gamma(1)}(1)\cap U_{\gamma},  B_{\gamma(1)}(2)\cap U_{\gamma},..., B_{\gamma(1)}(\gamma(1))\cap U_{\gamma}, U_{\gamma}\rangle\subset \langle W_1, ..., W_k\rangle={\bf W}.$$
\end{proof}

For the locally finite topology, we have the following result.

\begin{theorem}
Let $X$ be a metrizable space. If $(CL(X), \tau_{lf})$ has an $\omega^\omega$-base consisting of basic neighborhoods, then the boundary of each closed subset of $X$ is $\sigma$-compact.
\end{theorem}

\begin{proof}
Fix any $A\in CL(X)$. We prove that $\partial A$ is $\sigma$-compact. Without loss of generality, we may assume that $\partial A$ is nonempty. By Lemma~\ref{l2}, it suffices to prove that $\partial A$ has an $\omega^\omega$-base in $X$. Since $\{\partial A\}\in CL(X)$, we conclude that $(CL(X), \tau_V)$ has an $\omega^\omega$-base $\{{\bf U}_\alpha: \alpha\in \omega^\omega\}$ at $\{\partial A\}$, which is consist of basic neighborhoods. For each $\alpha\in\omega^\omega$, assume that $${\bf U}_\alpha=\bigcup\{U_{\gamma, \alpha}^{+}\cap\mathcal{V_{\gamma, \alpha}^{-}}: \gamma\in \Gamma_\alpha\},$$where each $U_{\gamma, \alpha}$ is open in $X$, each $\mathcal{V_{\gamma, \alpha}}$ is a locally finite family of open subsets of $X$, and each $U_{\gamma, \alpha}^{+}\cap\mathcal{V_{\gamma, \alpha}^{-}}\neq\emptyset;$ let $$U_\alpha=\bigcup\{\bigcup(U_{\gamma, \alpha}\cap\bigcup\mathcal{V_{\gamma, \alpha}}): \gamma
\in \Gamma_{\alpha}\}.$$ In order to complete the proof, it need to prove the following Claim 1.

\smallskip
{\bf Claim 1.} The family $\{U_\alpha: \alpha\in  \omega^\omega\}$ is an $\omega^\omega$-base for $\partial A$ in $X$.

\smallskip
Indeed, take any $\alpha, \beta\in\omega^\omega$ with $\alpha<\beta$. Then since ${\bf U}_\beta\subset {\bf U}_\alpha$, we conclude that $U_\beta\subset U_\alpha$. Suppose not, there exists $x\in U_\beta\setminus U_\alpha$ such that $x$ belongs to for some $V'\in \mathcal{V}_{\gamma, \beta}$ and $x\in V'\cap U_{\gamma, \beta}$, where $\gamma\in\Gamma_{\beta}$. For any $V\in \mathcal{V}_{\gamma, \beta}$, since $U_{\gamma, \beta}^{+}\cap\mathcal{V_{\gamma, \beta}^{-}}\neq\emptyset,$ we can pick $x_V\in V\cap U_{\gamma, \beta}$;  let $B=\{x\}\cup\{x_V: V\in \mathcal{V}_{\gamma, \beta}\}$. Since $\mathcal{V}_{\gamma, \beta}$ is locally finite, it follows that $B$ is closed in $X$, hence $B\in CL(X)$. However, $B\in {\bf U}_\beta\setminus {\bf U}_\alpha$, which is a contradiction.

Let $\partial A\subset V$ with $V$ open in $X$. Then $\langle V\rangle$ is an open neighborhood of $\{\partial A\}$ in $(CL(X), \tau_{lf})$, hence there exists an $\alpha\in\omega^\omega$ such that $\{\partial A\}\in {\bf U}_\alpha\subset \langle V\rangle$. Then it is straightforward to prove that $U_\alpha\subset V$.
\end{proof}

The following question is still unknown for us.

\begin{question}
Let $X$ be a metrizable space. If the boundary of each closed subset of $X$ is $\sigma$-compact, does $(CL(X), \tau_{lf})$ have an $\omega^\omega$-base?
\end{question}

Now we prove the third main theorem, which gives an answer to Question~\ref{q1}. First, we give some definitions.

A subset $A$ of a poset $P$ {\it dominates} a subset $B\subset P$ if for each $b\in B$ there is $a\in A$ with $b\leq a$.
Given two posets $P, Q$, we say that $Q$ is {\it Tukey reducible} to $P$ (denote by $Q\leq_T P$) if there exists a map $f: P\to Q$ such that for every confinal subset $C\subset P$ its image $f(C)$ is confinal in $Q$.

\begin{theorem}\cite{Ch1974}\label{th}
(Christensen) A metrizable space $X$ is Polish if and only if $\omega^\omega\succeq \mathcal{K}(X)$ if and only if $\mathcal{K}(X)\leq_T\omega^\omega$.

\end{theorem}

\begin{theorem}\label{tt2}
Let $X$ be a metrizable space. Then $(CL(X), \tau_F)$ has an $\omega^\omega$-base consisting of basic neighborhoods if and only if $X$ is a Polish space.
\end{theorem}

\begin{proof}
Sufficiency. By Theorem ~\ref{th}, we only need to prove $\mathcal{K}(X)\leq_T\omega^\omega$.

Fix any $x\in X$. Since $(CL(X), \tau_F)$ has an $\omega^\omega$-base $$\mathcal{B}=\{(\mathcal{U}_\alpha)^-\cap (K_\alpha^c)^+: \alpha\in \omega^\omega\}$$ (consisting of basic neighborhoods) at $\{x\}$, where each $\mathcal{U}_\alpha$ is a finite family of open subsets of $X$ and each $K_\alpha$ is compact in $X$, we can define a map $$f_1: \omega^\omega\to \mathcal{K}(X\setminus \{x\})\ \mbox{by}\ f_1(\alpha)=K_\alpha\ \mbox{for each}\ \alpha\in\omega^\omega.$$ Then, for any $\alpha, \beta\in\omega^\omega$ with $\alpha\leq\beta$, we have $f_1(\alpha)\subset f_1(\beta)$. Indeed, since $$(\mathcal{U}_\beta)^-\cap (K_\beta^c)^+\subset (\mathcal{U}_\alpha)^-\cap (K_\alpha^c)^+,$$ it follows that $K_\alpha\subset K_\beta$, that is, $f_1(\alpha)\subset f_1(\beta)$. Moreover, we conclude that $\{K_\alpha: \alpha\in \omega^\omega\}$ dominates $\mathcal{K}(X\setminus \{x\})$. In fact, for a compact subset $K\subset X\setminus \{x\}$, it is obvious that $(K^c)^+$ is a neighborhood of $\{x\}$ in $(CL(X), \tau_F)$, then there is a ${\bf B}_\alpha\in\mathcal{B}$ such that $${\bf B}_\alpha=(\mathcal{U}_\alpha)^-\cap (K_\alpha^c)^+\subset (K^c)^+,$$ hence $K\subset K_\alpha$. Therefore, $\{K_\alpha: \alpha\in \omega^\omega\}$ dominates $\mathcal{K}(X\setminus \{x\})$.

Let $V$ be an open neighborhood of $x$ in $X$ such that $\overline{V}\neq X$. Pick any $y\in X$ with $y\notin \overline{V}$, and let $Y=\{y\}\cup \overline{V}$. Then $(CL(Y), \tau_F)$ has an $\omega^\omega$-base $$\mathcal{C}=\{(\mathcal{W}_\alpha)^-\cap (H_\alpha)^+: \alpha\in\omega^\omega\}$$ consisting of basic neighborhoods at $\{y\}$ in $(CL(Y), \tau_F)$, where each $\mathcal{W}_\alpha$ is a finite family of open subsets of $Y$ and each $H_\alpha$ is compact in $\overline{V}$. Define a map $$f_2: \omega^\omega\to \mathcal{K}(\overline{V})\ \mbox{by}\ f_2(\alpha)=H_\alpha\ \mbox{for each}\ \alpha\in\omega^\omega.$$ Take any $\alpha, \beta\in \omega^\omega$ with $\alpha\leq \beta$; then $f_2(\alpha)\subset f_2(\beta)$. Moreover, $\{H_\alpha: \alpha\in\omega^\omega\}$ dominates $\mathcal{K}(\overline{V})$ by a similar proof above. Now define a map $$f: \omega^\omega \to \mathcal{K}(X)\ \mbox{by}\ f(\alpha)=f_1(\alpha)\cup f_2(\alpha)\ \mbox{for each}\ \alpha\in\omega^\omega.$$ Then $f$ is the desired map. Indeed, let $P$ be a confinal subset of $\omega^\omega$. We prove $f(P)$ is confinal in $\mathcal{K}(X)$. Fix any $K\in \mathcal{K}(X)$; then $$K\cap (X\setminus V)\in \mathcal{K}(X\setminus\{x\}), K\cap \overline{V}\in\mathcal{K}(\overline{V})\ \mbox{and}\ K=(K\cap (X\setminus V))\cup (K\cap\overline{V}).$$ Since $K\cap (X\setminus V)\in \mathcal{K}(X\setminus \{x\})$, $K\cap \overline{V}\in \mathcal{K}(\overline{V})$, there exist $\alpha_1, \alpha_2\in\omega^\omega$ such that $K\cap(X\setminus V)\subset f_1(\alpha_1)$, $K\cap\overline{V}\subset f_2(\alpha_2)$. Since $P$ is a confinal subset of $\omega^\omega$, there is a $\beta\in P$ with $\alpha_1<\beta, \alpha_2<\beta$, hence $$f(\beta)=f_1(\beta)\cup f_2(\beta)\supset f_1(\alpha_1)\cup f_2(\alpha_2)\supset (K\cap (X\setminus V))\cup (K\cap \overline{V})=K.$$ Therefore, $f(P)$ is a confinal subset of $\mathcal{K}(X)$, thus $\mathcal{K}(X)\leq_T\omega^\omega$.

\smallskip
Necessity. Fix any $A\in CL(X)$. We construct an $\omega^\omega$-base for $A$ in $(CL(X), \tau_F)$.
Since $X$ is Polish, it follows that $X\setminus A$ is a Polish space by \cite[Theorem 4.3.23]{E1989}. By Theorem ~\ref{th}, $\mathcal{K}(X\setminus A)\leq_T\omega^\omega$. Let $g: \omega^\omega\to \mathcal{K}(X\setminus A)$ be the map witness the Tukey reduction. Let $D=\{a_n: n\in\omega\}$  be a countable dense subset of $A$. For each $n, k\in\mathbb{N}$, let $V_k(n)=\{x\in X: d(x, a_n)\leq 1/k\}$, where $d$ is the metric on $X$. For each $\alpha\in \omega^\omega$, define $\mathcal{U}_\alpha=\{V_k(n): k= \alpha(0), n\leq \alpha(0)\}$. For any $\alpha\in \omega^\omega$, pick $\alpha'\in \omega^\omega$ with $\alpha'(i)=\alpha(i+1)$ for each $i\in\omega$, and let ${\bf B}_\alpha=(\mathcal{U}_\alpha)^-\cap (g(\alpha')^c)^+$. We claim that $\mathcal{B}=\{{\bf B}_\alpha: \alpha\in\omega^\omega\}$ is an $\omega^\omega$-base at $A$ in $(CL(X), \tau_F)$.

First, we prove $\mathcal {B}$ is a base at $A$ in $(CL(X), \tau_F)$. Let $(\mathcal{W})^-\cap (K^c)^+$ be a neighborhood of $A$, where $\mathcal{W}=\{W_i: i\leq m\}$ is a finite family of open subsets of $X$ with $W_i\cap A\neq \emptyset$, $K$ is a compact subset of $X$ with $K\cap A=\emptyset$. Choose $\beta\in \omega^\omega$ such that $K\subset g(\beta)\subset X\setminus A$. For each $i\leq m$, pick $a_{n_i}\in W_i\cap D$; then there exists $k_i\in\mathbb{N}$ such that $V_{k_i}(n_i)\subset W_i\cap (X\setminus K)$. Now let $p=\max\{k_i, n_i: i\leq m\}$, and let $\mathcal{V}=\{V_p(i): i\leq p\}$. Then $$A\in (\mathcal{V})^-\cap (g(\beta)^c)^+\subset (\mathcal{W})^-\cap (K^c)^+.$$ Pick $\gamma\in \omega^\omega$ such that $\gamma(0)=p$ and $\gamma'=\beta$. Then $${\bf B}_\gamma=(\mathcal{U}_\gamma)^-\cap (g(\gamma')^c)^+=(\mathcal{V})^-\cap (g(\beta)^c)^+\in \mathcal{B},$$ hence ${\bf B}_\gamma\subset (\mathcal{W})^-\cap (K^c)^+$.

Now take any $\alpha, \beta\in\omega^\omega$ with $\alpha\leq \beta$; then $\alpha(0)\leq \beta(0)$, $\alpha'\leq\beta'$. Each element of $\mathcal{U}_\alpha$ contains some element in $\mathcal{U}_\beta$, and $ K_{\alpha'}\subset K_{\beta'}$. Therefore, $${\bf B}_\beta=(\mathcal{U}_\beta)^-\cap (g(\beta')^c)^+\subset (\mathcal{U}_\alpha)^-\cap (g(\alpha')^c)^+={\bf B}_\alpha.$$

Thus $\mathcal{B}$ is an $\omega^\omega$-base at $A$ in $(CL(X), \tau_F)$.
\end{proof}

However, the following question is still unknown for us.

\begin{question}
If $(CL(X), \tau_F)$ has an $\omega^\omega$-base at $A\in CL(X)$, does $(CL(X), \tau_F)$ have an $\omega^\omega$-base at $A$ consisting of basic neighborhoods?
\end{question}

\begin{remark}
(1) By Corollary~\ref{c0}, Theorems~\ref{tt1} and~\ref{tt2}, $F(\mathbb{R})$, $A(\mathbb{R})$, $(CL(\mathbb{R}), \tau_V)$ and $(CL(\mathbb{R}), \tau_F)$ all have $\omega^\omega$-bases, where $\mathbb{R}$ is the real number space.

(2) It is well known that $\mathbb{R}^{\omega}$ is a polish space and not a $k_{\omega}$-space. From Theorems \ref{tt1} and~\ref{tt2}, it follow that $(CL(\mathbb{R}^{\omega}), \tau_F)$ has an $\omega^\omega$-base and $(CL(\mathbb{R}^{\omega}), \tau_V)$ does not have an $\omega^\omega$-base.

(3) It is well known that $\mathbb{Q}$ is a $\sigma$-compact space and not a polish space. From Theorems \ref{tt1} and~\ref{tt2}, it follows that $(CL(\mathbb{Q}), \tau_F)$ does not have an $\omega^\omega$-base consisting of basic neighborhoods and $(CL(\mathbb{Q}), \tau_V)$ has an $\omega^\omega$-base.

(4) For a metrizable space $X$, it follow from \cite[Theorems 7.6.30 and 7.6.36]{AT2008}, \cite[Theorem 4.6]{LRZ2020} and Theorem~\ref{tt3} that $A(X)$ has an $\omega^\omega$-base if $F(X)$ has an $\omega^\omega$-base.
\end{remark}

In \cite{GKL2015}, Gabriyelyan, K\c{a}kol and Leiderman proved that each Fr\'{e}het-Urysohn topological group having an $\omega^\omega$-base is first-countable. Next, we shall show that if $(CL(X), \tau_F)$ is a Fr\'{e}het-Urysohn space with an $\omega^\omega$-base, then it is first-countable. First, we give some technical lemmas.

Let $X_n=\{x_{ni}: i\in\omega\}$ for each $n\in\omega$ such that $X_n\cap X_m=\emptyset$ for any distinct $n, m\in\omega$, and put $$J(\omega)=\{x\}\cup(\bigcup_{n\in\omega} X_n),$$ endowed with a topology on $J(\omega)$ as follow: Each $x_{ni}$ is an isolated point for any $n, i\in \omega$, and the neighborhood form of $x$ is defined as $\{x\}\cup (\bigcup_{k\leq n}X_n)$, where $k\in \omega$. Then $J(\omega)$ is a second countable space.

\begin{lemma}\label{l11}
$(CL(J(\omega)), \tau_F)$ is not Fr\'echet-Urysohn.
\end{lemma}

\begin{proof}
Suppose $(CL(J(\omega)), \tau_F)$ is Fr\'echet-Urysohn.
Then, for each $n\in\omega$, let $$A_{n, i}=\{x_{0j}: j\leq n\}\cup \{x_{ni}\}.$$ Now put $\mathcal{A}=\{A_{n, i}: n\in\mathbb{N}, i\in \omega\}.$ We claim that $X_0$ belongs to the closure of $\mathcal{A}$ in $(CL(J(\omega)), \tau_F)$.

Indeed,take any open neighborhood ${\bf U}=\mathcal{U}^-\cap (K^c)^+$ of $X_0$ in $(CL(J(\omega)), \tau_F)$, where $\mathcal{U}=\{U_i: i\leq k\}$ a finite family of open subsets of $J(\omega)$ and $K$ is a compact subset of $J(\omega)$. Pick $y_i\in X_0\cap U_i$ for any $i\leq k$; then there is $m\in\omega$ such that $\{y_i: i\leq k\}\subset \{x_{0i}: i\leq m\}$, Let ${\bf V}=\{x_{01}\}^-\cap ...\cap\{x_{0m}\}^-\cap (K^c)^+$; then ${\bf V}$ is a neighborhood of $X_0$ and ${\bf V}\subset {\bf U}$. We conclude that ${\bf V}$ contains some $A_{m, i}$. In fact, since $K$ is compact, it follows that $K\cap X_m$ is finite; hence we can pick $x_{mi_0}\in X_m\setminus K$. Thus $A_{m, i_0}\in {\bf V}\subset {\bf U}$. Therefore, $X_0$ belongs to the closure of $\mathcal{A}$ in $(CL(J(\omega)), \tau_F)$.

Since $(CL(J(\omega)), \tau_F)$ is Fr\'echet-Urysohn, there is a sequence $\{B_p: p\in\omega\}\subset \{A_{n, i}: n\in\mathbb{N}, i\in \omega\}$ such that $B_p\to X_0$ as $p\rightarrow\infty$. Put $$\mathbb{N}_{1}=\{n\in\omega: B_p\cap X_n\neq\emptyset\ \mbox{for some}\ p\in\omega\}.$$ We prove that $\mathbb{N}_{1}$ is an infinite subset of $\omega$.

Suppose not, there is $m_{0}\in\omega$ such that $B_p\subset \bigcup_{i\leq m_{0}}X_i$ for any $p\in\omega$, which implies that $\{B_p: p\in\omega\}\subset\{A_{n, i}: n\leq m_{0}, i\in \omega\}$. Choose any $x_{0q}\in X_0$, where $q>m_{0}$. Let ${\bf U}_1=\{x_{0q}\}^-$. Then ${\bf U}_1$ is a neighborhood of $X_0$, and ${\bf U}_1\cap \{B_p: p\in\omega\}=\emptyset$, which is a contradiction. Hence $\mathbb{N}_{1}$ is an infinite subset of $\omega$.
Therefore, there is a subsequence $\{B_{p_j}: j\in \omega\}$ of $\{B_p: p\in\omega\}$ such that $p_{0}>0$,  $\{p_j: j\in\omega\}$ is increasing and $x_{n_{p_j}i_{p_j}}\in X_{n_{p_j}}\cap B_{p_j}$ for any $j\in\omega$. Let $K=\{x\}\cup\{x_{n_{p_j}i_{p_j}}: j\in\omega\}$; then $K$ is a compact subset of $J(\omega)$ that doesn't meet $X_0$. Then $(K^c)^+$ is a neighborhood of $X_0$ and $B_{p_j}\notin (K^c)^+$ for each $j\in\omega$. This is a contradiction.

Therefore, $(CL(J(\omega)), \tau_F)$ is not Fr\'echet-Urysohn.
\end{proof}

\begin{lemma}\label{l12}
$(CL(S_\omega), \tau_F)$ is not Fr\'echet-Urysohn.
\end{lemma}

\begin{proof}
 Suppose $(CL(S_\omega), \tau_F)$ is Fr\'echet-Urysohn. Let $$S_\omega=\{x\}\cup\{x_i(n): i, n\in\omega\},$$ for each $n\in\omega$, $x_i(n)\to x$ as $i\to\infty$. Let $A=\{x\}\cup\{x_i(0): i\in\omega\}$. Then $A\in CL(S_\omega)$; let $A_{n, j}=\{x_j(n)\}\cup\{x_i(0): i\leq n\}$ for any $j, n\in\omega$. We claim that $A\in Cl\{A_{n, j}: n, j\in\omega\}$.

 Indeed, take any basic open neighborhood ${\bf V}=\cap_{j\leq k}U_j^-\cap (K^c)^+$ of $A$, where each $U_{j}$ is open and $K$ is compact in $S_\omega$. We prove that ${\bf V}\cap\{A_{n, j}: n, j\in\omega\}\neq\emptyset$. Clearly, $A\cap K=\emptyset$; for each $j\leq k$, pick $x_{i_j}(0)\in U_j\cap (A\setminus\{x\})$. Then $${\bf U}=\cap_{j\leq k}\{x_{i_j}(0)\}^-\cap (K^c)^+$$ is an open neighborhood of $A$ and ${\bf U}\subset {\bf V}$. Since $K$ is compact and $x\notin K$, there is $m\in\omega$ such that $K\cap \{x_i(n): i\geq m, n\in\mathbb{N}\}=\emptyset$. Then $A_{n, m}\cap K=\emptyset$ for all $n\in\omega$. Pick any $p\geq \max\{i_j: j\leq k\}$; then $A_{p, m}\in {\bf U}$. Therefore, the point $A$ belongs to the closure of $\{A_{n, j}: n, j\in\omega\}$ in $(CL(S_\omega), \tau_F)$.

 Since $(CL(S_\omega), \tau_F)$ is Fr\'echet-Urysohn. There is a sequence $\{A_{n_r, j_{n_r}}: r\in\omega\}\subset \{A_{n, j}: n, j\in\omega\}$ converging to $A$, then $\{n_r: r\in\omega\}$ is unbounded. Indeed, if not, there is $q\in\omega$ such that $q> \max\{n_r: r\in\omega\}$, then $\bigcap_{i\leq q}\{x_i(0)\}^-$ is a neighborhood of $A$; however, $(\bigcap_{i\leq q}\{x_i(0)^-\})\cap \{A_{n_r, j_{n_r}}: r\in\omega\}=\emptyset$, which is a contradiciton.  Without loss of generality, we may assume $\{n_r: r\in\omega\}$ is increasing, otherwise we may choose an increasing subsequence. Choose $f\in\omega^\omega$ such that $f(n)=j_{n}+1$ for $n=n_r, r\in\mathbb{N}$, otherwise $f(n)=0$. Let $$V=\{x_i(n): i>f(n), n\in\omega\}\cup\{x\}.$$ Then $V$ is an open neighborhood of $x$ in $S_\omega$ and $V^-$ is a neighborhood of $A$ in $(CL(S_\omega), \tau_F)$. However, $V^-\cap \{A_{n_r, j_{n_r}}: r\in\omega\}=\emptyset$, which is an contradiction.

 Therefore, $(CL(S_\omega), \tau_F)$ is not Fr\'echet-Urysohn.
\end{proof}

A space $X$ is said to be {\it countable tightness} if, for each subset $A$ and $x\in \overline{A}$, there exists a countable subset $C$ of $A$ such that $x\in \overline{C}.$ Moreover, a space $X$ is called {\it Lindel\"of} if each open cover of $X$ has a countable subcover.

\begin{lemma}\label{l13}
For a space $X$, if $(CL(X), \tau_F)$ is of countable tightness, then $X$ is Lindel\"of.
\end{lemma}

\begin{proof}
Let $\mathcal{U}$ be an any open cover of $X$. Fix any point $x\in X$, and let $$\mathcal{U}'=\{U\setminus\{x\}: U\in\mathcal{U}\}.$$ Then $\mathcal{U}'$ is an open cover of $X\setminus \{x\}$. Put $$\mathcal{V}=\{\bigcup\mathcal{W}: \mathcal{W}\in\mathcal{U'}^{<\omega}\}.$$ For any compact subset $H$ of $X\setminus\{x\}$, we can choose a fixed $V_H\in\mathcal{V}$ such that $H\subset V_H$. Let $\mathcal{F}=\{X\setminus V: V\in\mathcal{V}\}$; we claim that $\{x\}$ belongs to the closure of $\mathcal{F}$ in $(CL(X), \tau_F)$. Indeed, let $\bigcap_{i\leq n}V_i^-\cap (K^c)^+$ be a basic neighborhood of $\{x\}$, where each $V_{i}$ is open and $K$ is compact. Clearly,
$x\notin K$ and $K\subset V_K$, then $x\in X\setminus V_K\subset X\setminus K$, hence $X\setminus V_K\in \bigcap_{i\leq n}V_i^-\cap (K^c)^+$. Therefore, $\{x\}$ belongs to the closure of $\mathcal{F}$ in $(CL(X), \tau_F)$.

Since $(CL(X), \tau_F)$ is of countable tightness, there is a countable subfamily $\mathcal{F}'\subset \mathcal{F}$ such that $\{x\}$ belongs to the closure of $\mathcal{F}'$ in $(CL(X), \tau_F)$. Let $\mathcal{V'}=\{V: X\setminus V\in\mathcal{F}'\}$; we conclude that $\mathcal{V}'$ is a cover of $X\setminus \{x\}$. Indeed, for any $y\in X\setminus \{x\}$, the set $(X\setminus \{y\})^+$ is an open neighborhood of $\{x\}$, hence there is $F\in\mathcal{F}'$ such that $F\in (X\setminus \{y\})^+$, then $F\subset X\setminus \{y\}$, that is, $\{y\}\subset X\setminus F\in \mathcal{V}'$. Hence $\mathcal{V}'$ is a cover of $X\setminus \{x\}$.

For each $V\in\mathcal{V}'$, there are finitely many elements $\mathcal{U}_V\subset \mathcal{U}$ such that $V\subset \bigcup\mathcal{U}_V$. Pick any $U\in\mathcal{U}$ with $x\in U$. Then $\{U\}\cup (\bigcup\{\mathcal{U}_V: V\in\mathcal{V}'\})$ is a countable subcover of $\mathcal{U}$, hence $X$ is Lindel\"of.
\end{proof}

A space $X$ is said to be {\it strongly
Fr\'echet-Urysohn} if the following condition is satisfied: for any countable family $\{A_n: n\in \mathbb{N}\}$ of subsets of $X$ with $x\in \overline{A_n}$ for each $n$, there exists $x_n\in A_n$ such that $x_n \to x$. A space $X$ is strongly
Fr\'echet-Urysohn if and only if it is Fr\'echet-Urysohn and contains no copy of $S_{\omega}$, see \cite{LY2016}.

\begin{theorem}\label{tt4}
For a space $X$, the hyperspace $(CL(X), \tau_F)$ is a Fr\'echet-Urysohn space with an $\omega^\omega$-base if and only if $(CL(X), \tau_F)$ is first-countable.
\end{theorem}

\begin{proof}
The necessity is obvious. It suffices to prove the sufficiency. Suppose $(CL(X), \tau_F)$ is a Fr\'echet-Urysohn space with an $\omega^\omega$-base, hence $(CL(X), \tau_F)$ is of countable tightness, then $(CL(X), \tau_F)$ is  Lindel\"of by Lemma~\ref{l13}, which implies that $X$ is Lindel\"of. Moreover, it is obvious that $X$ is a Fr\'echet-Urysohn space with an $\omega^\omega$-base, hence $X$ has a countable $cs^*$-network at each $x\in X$ by \cite[Theorem 6.4.1]{B2017}. By Lemma ~\ref{l12}, $X$ contains no closed copy of $S_\omega$, hence $X$ is strongly Fr\'echet-Urysohn. It is straightforward to prove that each strongly Fr\'echet-Urysohn space with a countable $cs^*$-network at each point $x\in X$ is a first-countable space. Fix any $x\in X$, and let $\{V_n: n\in\omega\}$ be a decreasingly local base at $x$ with $\overline{V}_{n+1}\subset V_n$ for each $n\in\omega$. Now it suffices to prove that there exists $n\in\omega$ such that $\overline{V_{n}}$ is countably compact. Indeed, if some $\overline{V_{n}}$ is countably compact, then $X$ is locally compact at point $x$ since $X$ is Lindel\"of. By takeing arbitrary of $x$, we conclude $X$ is locally compact. Hence, by \cite[Theorem 11]{H1998}, $(CL(X), \tau_F)$ is first-countable. The result is desired.

Suppose that none of $\overline{V}_n$ is countably compact, then there exists an increasing subsequence $\{n_j: j\in\omega\}\subset \omega$ such that each $V_{n_j}$ contains a countable closed discrete subset $\{x_{n_j i}: i\in\omega\}$ and $V_{n_{j+1}}\cap \{x_{n_j i}: i\in\omega\}=\emptyset$. Put $Y=\{x\}\cup\{x_{n_j i}: j, i\in\omega\}$. Then $Y$ is a closed copy of $J(\omega)$. Since $Y\subset X$ and $(CL(X), \tau_F)$ is a Fr\'echet-Urysohn space, it follows that $(CL(Y), \tau_F)$ is Fr\'chet-Urysohn, which is a contradiction with Lemma ~\ref{l11}.
\end{proof}

A space $X$ is hemicompact if there is a countable family of compact subsets $\{K_i: i\in \omega\}$ such that for any $K\in \mathcal{K}(X)$, there is $n\in\omega$ with $K\subset K_n$.

\begin{theorem}
Let $X$ be a metrizable space. Then the following statements are equivalent:
\begin{enumerate}
\item $(CL(X), \tau_F)$ is a Fr\'echet-Urysohn space with an $\omega^\omega$-base.

\smallskip
\item $(CL(X), \tau_F)$ is first-countable.

\smallskip
\item $X$ is a locally compact and second countable space.
\end{enumerate}
\end{theorem}

\begin{proof}
By Theorem~\ref{tt4}, we have (1) $\Leftrightarrow$ (2). (3) $\Rightarrow$ (2) by \cite[Theorem 5.1.5]{B1993}.

(2) $\Rightarrow$ (3). Since $(CL(X), \tau_F)$ is first-countable, it follows that $X$ is first-countable, hemicompact, separable by \cite[Corollary 7]{H1998}. By \cite[3.4.E]{E1989}, each first-countable, hemicompact space is locally compact, hence $X$ is locally compact and second countable.
\end{proof}

\end{document}